\documentclass[10pt]{amsart}
\usepackage{amsmath}
\usepackage{amssymb} 
\usepackage{epsbox}
\usepackage{amsthm}
\usepackage{graphicx}
\usepackage{algorithm}
\usepackage{algorithmic}

 \newtheorem{thm}{Theorem}
 \newtheorem{cor}{Corollary}
 
 \newtheorem{lem}{Lemma}
 {\theoremstyle{definition}
 \newtheorem{rem}{Remark}}
 {\theoremstyle{definition}
  \newtheorem{defn}{Definition}}
 {\theoremstyle{definition}
 \newtheorem{exam}{Example}}
 {\theoremstyle{definition}
 }
 \newcommand{\mP}{\mathcal{P}}
 \newcommand{\mS}{\mathcal{S}}
  \newcommand{\mR}{\mathcal{R}}
\newcommand{\ba}{\mathbf{a}}
\newcommand{\bg}{\mathbf{g}}

 \newcommand{\C}{\mathbb{C}}

\title{Hurwitz equivalences of positive group generators}
\author{Tetsuya Ito}
\address{Graduate School of Mathematical Science, University of Tokyo, 3-8-1 Komaba Meguro-ku Tokyo 153-8914, Japan}
\email{tetitoh@ms.u-tokyo.ac.jp}

\begin{document}

\begin{abstract} 
For a positively presented group $G$, we provide a criterion for two tuples of positive group generators of $G$ to be Hurwitz equivalent or Hurwitz-conjugation equivalent. We also present an algorithmic approach to solve the Hurwitz equivalence and the Hurwitz search problems by using the word reversing method. 
\end{abstract}
 \maketitle

\section{Introduction}

Let $B_{n}$ be the braid group of $n$-strings and $\sigma_{1},\ldots,\sigma_{n-1}$ be the standard generators.
For a group $G$, we denote by $G^{n}$ the $n$-fold direct product of $G$, and we call an element of $G^{n}$ a {\it $G$-system} of length $n$.
For a fixed positive presentation $\mP=\langle \mS \:|\: \mR \rangle$ of $G$ we call an element $\mS^{m}$, a $G$-system consisting of positive generators $\mS$, a {\it generator $G$-system}.

The {\it Hurwitz action} is a right action of $B_{n}$ on $G^{n}$ defined by  
\[(g_{1},g_{2},\ldots,g_{n})\cdot \sigma_{i} = (g_{1},g_{2},\ldots, g_{i-1},g_{i+1},g_{i}^{g_{i+1}},g_{i+2},\ldots, g_{n})\]
where we denote $g_{i+1}^{-1}g_{i}g_{i+1}$ by $g_{i}^{g_{i+1}}$.
The Hurwitz action is diagrammatically represented as in Figure \ref{fig:hurwitzaction}. 
Two $G$-systems are called {\it Hurwitz equivalent} if they belong to the same orbit of the Hurwitz action.

\begin{figure}[htbp]
 \begin{center}
\includegraphics[width=60mm]{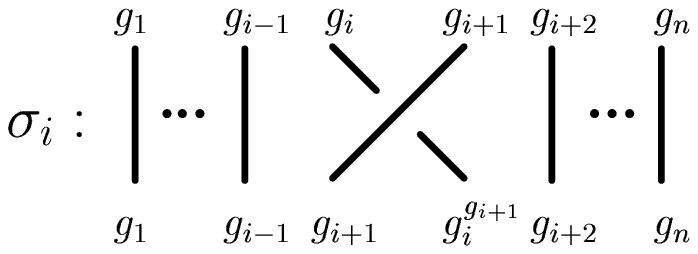}
 \end{center}
 \caption{Diagrammatic description of Hurwitz action }
 \label{fig:hurwitzaction}
\end{figure}

The group $G$ itself acts on $G^{n}$ from the right by conjugations, as 
\[ (g_{1},g_{2},\ldots,g_{n})\cdot g = (g_{1}^{g},g_{2}^{g},\ldots,g_{n}^{g}). \]
Two $G$-systems are said to be {\it conjugate} if they belong to the same $G$-orbit. The actions of $G$ and $B_{n}$ commute, so we regard the group $G \times B_{n}$ acting on $G^{n}$. We call this action the {\it Hurwitz-conjugation action} ({\it HC-action}, in short). Two $G$-systems $\bg$ and $\bg'$ are called {\it Hurwitz-conjugation equivalent} ({\it HC-equivalent}) if they belong to the same orbit of the HC-action. We denote by $\bg \sim_{H} \bg'$ (resp. $\bg \sim_{HC} \bg'$) if $\bg$ and $\bg'$ are Hurwitz (resp. HC-) equivalent.

In this paper we study the following two problems. 

\begin{description}
\item[Hurwitz equivalence problem] Given two $G$-systems, determine whether they are Hurwitz equivalent or not.
\item[Hurwitz search problem] Given two Hurwitz equivalent $G$-systems $\bg$ and $\bg'$, find a braid $\beta$ such that $\bg \cdot \beta = \bg'$.
\end{description}

These problems are very hard compared to the word and conjugacy problems. Liberman-Teicher showed that these problems are undecidable even for the braid groups \cite{lt}, which have various good properties and have nice solutions for the word and conjugacy problems. 

Although the Hurwitz equivalence/search problems are purely algebraic problems, they are closely related to geometry and topology.
By considering certain monodromy representations \cite{br}, many geometric objects in 4-dimensional topology and geometry such as braided surfaces \cite{ka}, Lefschetz fibrations \cite{ma}, and complex surfaces or complex curves \cite{kt},\cite{mt} are represented by a $G$-system for an appropriate group $G$. Such a $G$-system representative is not unique. Two $G$-system represent the same geometric object if and only if they are Hurwitz (or, HC-) equivalent. Thus, the Hurwitz equivalence/search problems are directly related to the classification problems these topological or geometric objects.

The aim of this paper is to propose an algebraic approach to the Hurwitz equivalence/search problem using theory of {\em word-reversing}.
We provide a criterion for two generator $G$-systems to be Hurwitz equivalent in Theorem \ref{thm:main}. Using this criterion, we give algorithmic approaches (Algorithm \ref{alg:htestsimple}, Algorithm \ref{alg:htest}) to solve the Hurwitz equivalence/search problems.

Unfortunately, our algorithms can be applied in very special cases and even worse, they are not deterministic and do not necessarily terminate in finite time. Nevertheless, our algorithm has several benefits.
First, in successful cases, our algorithm solves not only Hurwitz equivalence problems but also Hurwitz search problems. Moreover, one can also try to get stronger result, the classification of the Hurwitz equivalence classes of generator $G$-systems.
Second, in practice, one can apply our algorithms to try to show {\em arbitrary} $G$-systems are indeed Hurwitz equivalent, as we will discuss in Section 4.3.
Moreover, our algorithms can be implemented on a computer easily.
Finally, as for an application of geometry and topology, in many cases to show given $G$-systems are {\em not} Hurwitz equivalent is done by means of invariants of corresponding geometric objects. Thus, our algorithmic approach will provide a complementary method to studying such geometric objects.

The plan of the paper is as follows.
In Section 2, we review the theory of {\em word reversing} and {\em complete group presentations}.  We explain how the word reversing method solves the Hurwitz search problem in Section 3. We also present several applications, including results of HC-equivalences.
Based on the criterion in Section 3, we present algorithms to solve the Hurwitz equivalence/search problems in Section 4.
In Appendix we give an algorithm to try to show the embeddability of associated monoid, which allows us to try to classify generator $G$-systems.\\

\textbf{Acknowledgments.} 
The author is grateful to his advisor Toshitake Kohno for his encouragement and his comments. 
This research was supported by JSPS Research Fellowships for Young Scientists.

\section{Word reversing and complete presentation}

In this section we summarize the theory of word reversing and complete presentation. For details, see \cite{d2},\cite{d3}. Except Appendix, in this paper we only use right word reversing and right complete presentations, so we always drop the word ``right". 

Let $\mS=\{a_{1},\ldots,a_{m}\}$ be a finite set and $\mS^{*}$ be the free monoid generated by $\mS$. For a word $V \in \mS^{*}$ we denote the length of $V$ with respect to the generating set $\mS$ by $l(V)$.
A {\em positive relation} is a pair of elements in $\mS^{*}$, denoted by $W \equiv V$. 
A positive relation $W \equiv V$ is {\em homogeneous} if $l(V)=l(W)$. A positive relation of the form $aV \equiv aW$ or $Va \equiv Wa$ is called a {\it reducible relation}. As a group presentation, a reducible relation can be replaced by the simpler relation $V \equiv W$.

A {\it positive group presentation} is a group presentation of the form $\mP=\langle \mS \:|\: \mR \rangle$, where $\mR$ is a set of positive relations. Each positive relation $V \equiv W$ is understood as a group relation $V^{-1}W$. If both $\mS$ and $\mR$ are finite set, we say $\mP$ is a {\em finite positive presentation}.  We say $\mP$ is {\it homogeneous} if all relations are homogeneous. The {\it associated monoid} $M^{+}_{\mP}$ is a monoid $\mS^{*} \slash \equiv$, where $\equiv$ is the smallest congruence on $\mS^{*}$ that includes $\mR$.  

Now we introduce a word reversing, which is a fundamental tool to study positive presentation.

\begin{defn}[Word reversing]
Let $W$ and $W'$ be a word on $\mS \cup \mS^{-1}$. We say the word $W'$ is obtained from $W$ by performing one {\it word reversing} if one of the following holds.
\begin{enumerate}

\item $W'$ is obtained from $W$ by replacing a subword of the form $u^{-1}v$ with a subword $u' v'^{-1}$, where $u,v$ are nonempty words on $\mS$ and $u',v'$ are word on $\mS$ possibly an empty word, such that the positive relation $uu' \equiv vv'$ is contained in $\mR$.
\item $W'$ is obtained from $W$ by deleting a subword of the form $u^{-1}u$ where $u$ is a nonempty word on $\mS$. 
\end{enumerate}
\end{defn}

Diagrammatically, the word reversing is expressed as in Figure \ref{fig:reversing}.

\begin{figure}[htbp]
 \begin{center}
\includegraphics[width=70mm]{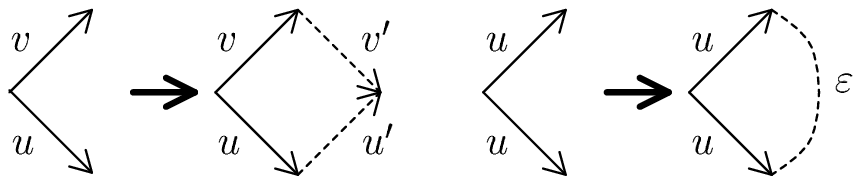}
 \end{center}
 \caption{Diagrammatic description of word reversing }
 \label{fig:reversing}
\end{figure}

We say a word $W$ on $\mS \cup \mS^{-1}$ is {\em reversible} to a word $W'$ on $\mS \cup \mS^{-1}$ if $W'$ is obtained from $W$ by iterated applications of word reversing operations. We denote by $W \curvearrowright W'$ if $W$ is reversible to $W'$.

For $u,v \in \mS^{*}$, $u^{-1}v \curvearrowright \varepsilon$ implies $u \equiv v$ \cite[Proposition 1.9]{d2}. Thus word reversing is used to show given two words are congruent. In fact if $u^{-1}v \curvearrowright \varepsilon$ then the word reversing not only shows $u$ and $v$ are congruent but also provides a Van-Kampen diagram of $(u,v)$, which contains more information about the congruence relation. 

Let $W,W' \in \mS^{*}$ be words representing the same element of $M_{\mP}^{+}$.
A {\em Van-Kampen diagram} of $(W,W')$ is an oriented sub-graph $D$ of the Cayley graph of $M_{\mP}^{+}$ which has the following properties.
\begin{enumerate}
\item $D$ has the unique source vertex which corresponds to an element $1$, and the unique sink vertex which corresponds to an element $W=W'$.
\item $D$ is a planer graph, and bounded by two edge paths defined by the word $W$ and $W'$. (In particular, $D$ defines a cellular decomposition $\mathcal{T}_{D}$ of a 2-disc).
\item The labeling of the boundary of each 2-cell in $\mathcal{T}_{D}$ is a relation in $\mR$. That is, the labeling is of the form $u^{-1}v$ and the relation $u \equiv v$ lies in $\mR$.
\end{enumerate}

See Figure \ref{fig:van-kampen} for example. Once a Van-Kampen diagram of $(W,W')$ is constructed, one can find a way to change the word $W$ into $W'$ by using the relations in $\mR$. That is, one can find a sequence of words on $\mS$
\[ W=W_{0} \rightarrow W_{1} \rightarrow \cdots \rightarrow W_{k-1} \rightarrow W_{k}=W' \]
where each $W_{i+1}$ is obtained from $W_{i}$ by performing a relation in $\mR$.

Recall the diagrammatic expression of word-reversing described in Figure \ref{fig:reversing}. Then the word reversing is considered as an operation to glue a 2-cell along paths $u^{-1}v$, or to identify two 1-cells having the same label. Thus, by expressing word reversing in a diagrammatic way, if $u^{-1}v \curvearrowright \varepsilon$ then we can draw a Van-Kampen diagram for $(u,v)$. 

\begin{exam}
\label{exam:van-kampen}
Let us consider a positive presentation of the braid group $B_{3}$
\[ \mP_{1}= \langle \mS \: | \: \mR \rangle = \langle x,y,z \: | \: xyx \equiv yxy , xy\equiv yz \equiv zx \rangle. \]
Here the relation $xy\equiv yz \equiv zx $ is understood as the three relations $xy\equiv yz$, $yz \equiv zx $ and $xy \equiv zx $. Let us reverse the word $(xxyx)^{-1}zxyz$.
\begin{eqnarray*}
x^{-1}y^{-1}x^{-1}\underline{x^{-1}z}xyx & \curvearrowright^{(1)} &   x^{-1}y^{-1}x^{-1}yx^{-1}xyz \\ 
x^{-1}y^{-1}x^{-1}y\underline{x^{-1}x}yz & \curvearrowright^{(2)} &
 x^{-1}y^{-1}x^{-1}yyz \\
x^{-1}\underline{y^{-1}x^{-1}y}yz & \curvearrowright^{(3)} &
 x^{-1}xy^{-1}x^{-1}yz \\
\underline{x^{-1}x}y^{-1}x^{-1}yz  & \curvearrowright^{(4)}& y^{-1}x^{-1}yz \\
\underline{y^{-1}x^{-1}yz} & \curvearrowright^{(5)}& \varepsilon. 
 \end{eqnarray*} 
 
According to this word reversing sequence, we attach a 2-cells or identify 1-cells, and obtain a Van-Kampen diagram of $(xxyx,zxyz)$ as shown in Figure \ref{fig:van-kampen}.
From this Van-Kampen diagram, we obtain a sequence of words
\[ xxyx \rightarrow xyxy \rightarrow zxxy \rightarrow zxyz\]
which converts the word $xxyx$ to $zxyz$ by using the relations in $\mR$.

\begin{figure}[htbp]
 \begin{center}
\includegraphics[width=90mm]{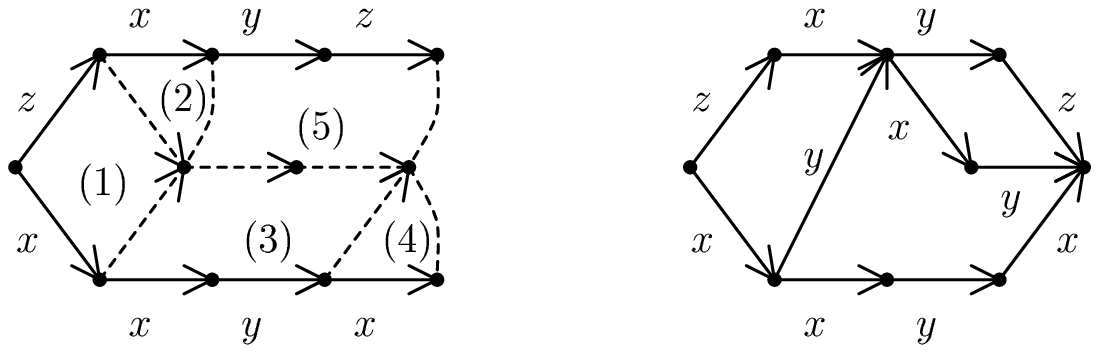}
 \end{center}
 \caption{Construction of Van-Kampen Diagram of $(xxyx,zxyz)$}
 \label{fig:van-kampen}
\end{figure}

\end{exam}

In general word reversing is not sufficient to show two words are congruent. That is, $u \equiv v$ does not imply $u^{-1}v \curvearrowright \varepsilon$. A {\em complete presentation} is a positive presentation such that the converse is true.

\begin{defn}[Complete positive group presentation]
A positive group presentation $\langle \mS \:| \:\mR \rangle $ is {\em complete} if 
$u^{-1}v \curvearrowright \varepsilon$ is equivalent to $u \equiv v$ for all 
$u,v \in \mS^{*}$.
\end{defn}

Thus, a complete presentation is a positive presentation such that word-reversing can detect congruence relations. There is a nice characterization of a complete presentation for a finite positive homogeneous presentation. This allows us to check whether a given homogeneous finite presentation is complete or not. 

\begin{thm}[\cite{d2}, Proposition 4.4]
A finite positive homogeneous presentation $\langle \mS \: | \:\mR \rangle $ is complete if and only if the condition $SC(\mS)$ (called the strong cube condition on $\mS$) holds.
\begin{itemize}
\item[$SC(\mS)$:] For $s,r,t \in \mS$ and $u,v \in \mS^{*}$, if $s^{-1}rr^{-1}t \curvearrowright  uv^{-1}$ then $(su)^{-1}(tv) \curvearrowright  \varepsilon. $
\end{itemize}
\end{thm}

Based on the strong cube condition, one can try to make a non-complete finite homogeneous positive presentation complete as follows. Assume that the strong cube condition fails for some $s,r,t,u,v$. That is, $s^{-1}rr^{-1}t \curvearrowright uv^{-1}$ but $(su)^{-1}(tv)\not\curvearrowright \varepsilon$. Then we add a new relation $su \equiv tv$ so that the strong cube condition is satisfied for such $s,r,t,u,v$. In general adding a new relation produces a new word reversing sequences, so the new presentation is not necessarily complete and we may iterate this operation. The precise algorithm is given as Algorithm \ref{alg:complete}. As we explained, this algorithm does not necessarily terminate.

\begin{algorithm}
\caption{:Presentation Completion Algorithm}
\label{alg:complete}
\begin{flushleft}
Input : A finite homogeneous positive presentation $\mP = \langle \mS\: | \:\mR \rangle$ of a group $G$.\\
Output: A complete presentation of $G$.\\
\end{flushleft}
\begin{enumerate}
\item Compute all pairs of words $u,v \in \mS^{*}$ such that $s^{-1}rr^{-1}t \curvearrowright uv^{-1}$ for some $s,r,t \in \mS$.
\item Check $(su)^{-1}(tv) \curvearrowright \varepsilon$ holds for all $u,v$ obtained by Step (1). If $(su)^{-1}(tv) \not\curvearrowright \varepsilon$, then replace the presentation $\mP$ with the new presentation 
\[ 
\langle \mS\: | \:\mR \cup \{ su \equiv tv\} \rangle
\]
and go back to Step (1).
\item Stop.
\end{enumerate}
\end{algorithm}

\begin{exam}
\label{exam:completion}

Let us consider the presentation of the braid group $B_{3}$ given by 
\[ \mP_{0}= \langle x,y,z \: | \: xyx \equiv yxy, xy \equiv yz \rangle.\]

We have $y^{-1}xx^{-1}y \curvearrowright xyx^{-1}z^{-1}$, but $(yxy)^{-1}yzx \not \curvearrowright \varepsilon$. Thus we add a new relation $yxy \equiv yzx$ to $\mP_{0}$ and 
 obtain the new presentation
\[ \mP'_{0}= \langle x,y,z \: | \: xyx \equiv yxy, xy \equiv yz, yxy \equiv yzx \rangle. \]
In $\mP'_{0}$, a new word reversing sequence $x^{-1}yy^{-1}x \curvearrowright (yxxy)(yxzx)^{-1}$ appears. Since there are no relations of the form $ z\cdots \equiv \cdots$, 
$(xyxxy)^{-1}(xyxzx) \not\curvearrowright  \varepsilon$. Thus, the presentation $\mP'_{0}$ is not complete. We need to add further relation $xyxzx \equiv xyxxy$, and so on. In this case, the completion procedure never terminate.

On the other hand, let us consider another presentation of $B_{3}$
\[ \mP_{1}= \langle x,y \: | \: xyx \equiv yxy , xy\equiv yz \equiv zx \rangle \]
used in Example \ref{exam:van-kampen}.
$\mP_{1}$ satisfies the strong cube conditions, so it is complete.
\end{exam}

\section{Hurwitz equivalence criterion via Coxeter elements and Coxeter words}

In this section we provide a criterion for Hurwitz equivalence and HC-equivalences.
To state our results, we introduce the notion of Hurwitz-compatible relations.

\begin{defn}
Let $\mP=\langle \mS \: | \: \mR \rangle$ be a finite homogeneous positive presentation of a group $G$ and $R:V \equiv W$ be a positive relation in $\mR$. For words $V = a_{1}a_{2}\cdots a_{l}$, $W = a'_{1}\cdots a'_{l}$ on $\mS$, let $\bg_{V}$, $\bg_{W}$ be generator $G$-systems defined by
\[ \bg_{V}=(a_{1},\ldots, a_{l}), \;\bg_{W}=(a'_{1},\ldots, a'_{l}). \]
We say a homogeneous positive relation $R$ is {\it Hurwitz-compatible} if there exists an $l$-braid $\beta_{R}$ such that $\bg_{V} \cdot \beta_{R} = \bg_{W}$. 
\end{defn}

By definition, Hurwitz compatible relations are homogeneous.
A typical example of a Hurwitz-compatible relation is a {\em word-conjugacy relation}, which is a positive relation of the form $R:aV \equiv Va'$, where $a,a' \in \mS$ and $V \in \mS^{*}$. In fact, $\bg_{aV} \cdot (\sigma_{1}\sigma_{2} \cdots \sigma_{l(V)}) = \bg_{Va'}$, thus we may choose $\beta_{R}=\sigma_{1}\sigma_{2}\cdots \sigma_{l(V)}$.

First of all, observe that there is an obvious and fundamental invariant of Hurwitz equivalence classes.
The {\em Coxeter element} (or, {\em the global monodromy}) of a $G$-system $\bg=(g_{1},\ldots,g_{m})$ is an element $C(\bg)=g_{1}g_{2}\cdots g_{m} \in G$. 
It is easy to see if $\bg \sim_{H} \bg'$ then $C(\bg) = C(\bg')$, and if $\bg \sim_{HC} \bg'$ then $C(\bg)$ and $C(\bg')$ are conjugate. 
The Coxeter element serves as a fundamental invariant to study Hurwitz equivalence class. For example, in \cite{i} the author classified $B_{3}$-systems having finite Hurwitz orbits by studying the centralizer of the Coxeter element.

For a generator $G$-system $\bg$, we can consider the refinement of the Coxeter element.
We call the word $g_{1}g_{2}\cdots g_{m} \in \mS^{*}$ the {\em Coxeter word} of $\bg$ and denote by $W(\bg)$. The Coxeter words contain more information than the Coxeter element itself.
 
\begin{lem}
\label{lem:key}
Let $G = \langle \mS \: | \: \mR \rangle$ be a positively presented group and $\ba = (a_{1},\ldots,a_{m})$, $\ba' = (a'_{1},\ldots,a'_{m})$ be generator $G$-systems of the same length. Let $W = a_{1}a_{2}\cdots a_{m}$ and $W=a_{1}'a_{2}'\cdots a_{m}'$ be the Coxeter words of $\ba$ and $\ba'$.
\begin{enumerate}
\item If $W'$ is obtained from $W$ by applying a Hurwitz-compatible relation $R:U \equiv V$ in $\mR$, then $\ba$ and $\ba'$ are Hurwitz equivalent.
\item If $W'$ is obtained from $W$ by applying a cycling operation $a_{1}a_{2}\cdots a_{m} \rightarrow a_{m}a_{1}a_{2}\cdots a_{m-1}$, then $\ba$ and $\ba'$ are HC equivalent.
\end{enumerate}
\end{lem}
\begin{proof}
Let us write $W = XUY$, $W' = XVY$ and $sh: B_{l} \rightarrow B_{l+k}$ be the $k$-fold shift map defined by $\sigma_{i} \rightarrow \sigma_{i+k}$ where $k=l(X)$ and $l=l(U)=l(V)$. Let $\iota: B_{l+k} \hookrightarrow B_{m}$ be the natural embedding of $B_{l+k}$.
Assume the relation $R:U \equiv V$ is Hurwitz-compatible, and let $\beta_{R}$ be an $l$- braid such that $\bg_{U} \cdot \beta_{R} = \bg_{V}$.
Then, $\ba \cdot \iota \circ sh(\beta_{R}) = \ba'$, thus $\ba$ and $\ba'$ are Hurwitz equivalent.
Similarly, assume that $W'$ is obtained from the cycling operation. Since $\ba \cdot (\sigma_{m-1}\sigma_{m-2}\cdots \sigma_{1}) = ( a_{m}, a_{1}^{a_{m}}, a_{2}^{a_{m}}, \ldots, a_{m-1}^{a_{m}})$, $\ba$ and $\ba'$ are HC-equivalent. 
\end{proof}

Theorem \ref{thm:main} below shows the relationships between word reversing and Hurwitz equivalences, and reveals that under some conditions the Coxeter element completely determines the Hurwitz equivalence class.

\begin{thm}
\label{thm:main}
Let $\mP = \langle \mS \: | \: \mR \rangle$ be a finite homogeneous positive presentation of a group $G$ such that $\mR$ consists of  Hurwitz-compatible relations. Let $\ba,\ba'$ be generator $G$-systems of the same length. 
\begin{enumerate}
\item If $W(\ba) \equiv W(\ba')$, then $\ba \sim_{H}\ba'$. 
\item If $W(\ba)^{-1}W(\ba') \curvearrowright \varepsilon$, then $\ba \sim_{H} \ba$. Moreover, in this case we can solve the Hurwitz search problem for $\ba$ and $\ba'$.
\item If $M_{\mP}^{+}$ injects in $G$, then $\ba \sim_{H} \ba'$ if and only if $C(\ba) = C(\ba')$ holds.
\item If $M_{\mP}^{+}$ injects in $G$ and the presentation $\mP$ is complete, then $\ba \sim_{H} \ba'$ if and only if $W(\ba)^{-1}W(\ba') \curvearrowright \varepsilon$. 
In such case, we can solve not only Hurwitz equivalence problem but also Hurwitz search problem.
\end{enumerate}
\end{thm}

\begin{proof}[Proof of Theorem \ref{thm:main}]
(1) directly follows from Lemma \ref{lem:key}. To prove (2), recall if $u^{-1}v \curvearrowright \varepsilon $, then we can obtain a Van-Kampen diagram for $(u,v)$ and   find a sequence of words on $\mS$
\[ W=W_{0} \rightarrow W_{1} \rightarrow \cdots \rightarrow W_{k-1} \rightarrow W_{k}=W' \]
where each $W_{i+1}$ is obtained from $W_{i}$ by performing the relations in $\mR$.
Let $\ba_{i}$ be the generator $G$-system of length $m$ whose Coxeter word is $W_{i}$.
Then by Lemma \ref{lem:key}, we can find a braid $\beta_{i}$ such that $\ba_{i} \cdot \beta_{i} = \ba_{i+1}$. Thus, $\ba \cdot (\beta_{0}\beta_{1}\cdots \beta_{k-1}) = \ba'$ so we solved the Hurwitz search problem.
To prove (3), observe that if the associated monoid $M_{\mP}^{+}$ embeds in $G$, then $C(\ba) = C(\ba')$ is equivalent to $W(\ba) \equiv W(\ba')$.
Finally, (4) follows from (2), (3) and the definition of the complete presentation.
\end{proof}

As we have given as Algorithm \ref{alg:complete}, for a finite homogeneous positive presentation one can try to check make the presentation complete. Moreover, using  Algorithm \ref{alg:embed} in Appendix,  one can also try to show $M_{\mP}^{+}$ injects in $G$ using the theory of word reversing. Thus, one can algorithmically try to show whether two generator $G$-systems are Hurwitz equivalent or not by using Theorem \ref{thm:main}. This point of view will be pursued in next section.

Theorem \ref{thm:main} is applied for some well-known families of groups.
The first example is an Artin group.
Let $M=(m_{ij})_{1 \leq i,j \leq m}$ be a Coxeter matrix, which is a symmetric matrix such that $m_{ii} = 1$ and $m_{ij} \in \{2,3,\ldots,\infty \}$ for distinct $i$ and $j$.
The {\it Artin group} $G$ corresponding to $M$ is a group defined by the positive presentation
\[ G = \langle a_{1},\ldots, a_{m} \: | \: R_{ij} \;\; (m_{i,j} \neq \infty) \rangle\]
 where $R_{ij}$ is a positive irreducible word conjugacy relation
 \[ R_{ij} : \underbrace{a_{i}a_{j}a_{i}\cdots }_{m_{ij}} \equiv \underbrace{a_{j}a_{i}a_{j}\cdots }_{m_{ij}}. \]
We call this presentation the {\it standard presentation}.
An Artin group is called {\it right-angled} if all $m_{ij}$ $(i\neq j)$ are either $2$ or $\infty$. A right-angled Artin group is represented by its {\it associated graph}, whose vertices are generators and two vertices $v_{i}$ and $v_{j}$ are connected by a single edge if and only if their corresponding generators $a_{i}$ and $a_{j}$ do not commute, in other words, if and only if $m_{ij}=m_{ji}= \infty$. 
 
\begin{cor}
\label{cor:artin}
Let $G$ be an Artin group with the standard presentation $\mP$ and $\ba$, $\ba'$ be generator $G$-systems having the same length. Then, 
\begin{enumerate}
\item $\ba \sim_{H} \ba'$ if and only if $C(\ba) = C(\ba')$.
\item If $G$ is a right-angled Artin group, then $\ba \sim_{HC}\ba'$ if and only if $C(\ba)$ and $C(\ba')$ are conjugate.
\end{enumerate}
\end{cor}
\begin{proof}
(1) follows from Theorem \ref{thm:main} (3) and the results of Paris \cite{p} that the associated monoid $M_{\mP}^{+}$ of the standard presentation $\mP$ of an Artin group $G$ injects in $G$. Similarly, the second statement follows from Lemma \ref{lem:key} and the fact that two conjugate elements of the same length in a right-angled Artin group are related by the cycling operations and the commutative relations, which are word conjugacy relations hence Hurwitz compatible \cite{cgw}.
\end{proof}

A generator $G$-system $\ba = (a_{1},\ldots,a_{m})$ is {\it full} if $\{a_{1},\ldots,a_{m}\} = \mS$ and $a_{i} \neq a_{j}$ for $i\neq j$.
The next corollary shows under some conditions, the HC-equivalence class of full generator $G$-systems are invariant under the permutation of its entries.

\begin{cor}
\label{cor:Rartin}
Let $G$ be a right-angled Artin group and $\Gamma$ be its associated graph. 
Then, all full generator $G$-systems are HC-equivalent if and only if $\Gamma$ is a forest.
\end{cor}

\begin{proof}
By Corollary \ref{cor:artin}, it is sufficient to show the associated graph $\Gamma$ is a forest if and only if all Coxeter elements of full generator systems are conjugate. With no loss of generality, we can assume that $\Gamma$ is connected.

First of all, assume $\Gamma$ is not a tree, thus there exists a simple edge-path which forms a loop. Let $a_{1},\ldots,a_{k}$ be generators of $G$ which correspond to the vertices of the loop. Then two full $G$-systems $(a_{1},a_{2},\ldots,a_{k},a_{k+1},a_{k+2},\ldots,a_{m})$ and $(a_{k},a_{k-1},\ldots,a_{1},a_{k+1},a_{k+2},\ldots,a_{m})$ have non-conjugate Coxeter elements, hence these two $G$-systems are not HC-equivalent.

Conversely, assume $\Gamma$ is a tree. Since the cycling operation preserves the conjugacy classes, it is sufficient to show $C=a_{1}a_{2}a_{3}\cdots a_{n-1}a_{n}$ is conjugate to $C'=a_{2}a_{1}a_{3}\cdots a_{n-1}a_{n}$. 
If there is no edge connecting $a_{1}$ and $a_{2}$, then $C=C'$. Thus, we assume that there is an edge $e$ which connects $a_{1}$ and $a_{2}$.
Let $\Gamma_{i}$ $(i=1,2)$ be the connected component of the graph $\Gamma - e$ which contains the vertex $v_{i}$.

Before proving $C$ and $C'$ are conjugate, we begin with a special case. Let us consider the right-angled Artin group $A_{4}=\langle a'_{1},a'_{2},a'_{3},a'_{4} \: | \: a'_{i}a'_{j} = a'_{j}a'_{i} \;\;(|i-j| \neq 1)\rangle$ whose associated graph is the $A_{4}$-Dynkin diagram. It is easy to see all Coxeter elements of full generator systems of $A_{4}$ are conjugate.

Now we proceed to general cases.
Let us denote the word $C$ as 
\[ C= a_{1}a_{2}W_{0} V_{1}W_{1}\cdots V_{k}W_{k} \]
where $W_{i}$ (resp. $V_{i}$) is the subword of $C$ which consists of the vertices of $\Gamma_{1}$ (resp. $\Gamma_{2}$). Put $W= W_{0}\cdots W_{k}$ and $V=V_{1}\cdots V_{k}$. Since $\Gamma_{1}$ and $\Gamma_{2}$ are disconnected, $W_{i}$ commutes with $V_{j}$.
Thus, we rewrite the word $C$ as $C= a_{1}a_{2}V W$
by using the commutative relations.
Let us consider the subgroup $H$ generated by $a_{1},a_{2},V,W$ and the map $A_{4} \rightarrow H$ defined by $a'_{1}\rightarrow W$, $a'_{2} \rightarrow a_{1}$, $a'_{3} \rightarrow a_{2}$, and $a'_{4} \rightarrow V$. This defines a group homomorphism, hence by using the result on $A_{4}$-case, we conclude that $C = a_{1}a_{2}VW$ is conjugate to $ C' = a_{2}a_{1}VW $.
\end{proof}

Finally, we give another example, {\em Garside groups}. 
A {\em Garside group} is a tuple $(G,M,\Delta)$ satisfying some axioms where $G$ is a group, $M$ is a submonoid of $G$, and $\Delta \in M$. 
We do not give a precise definition of Garside groups. For details see \cite{bgm} or \cite{d1} for example.
We use the following known properties of Garside groups.

\begin{enumerate}
\item The monoid $M$ embeds into $G$.

\item Let $A=\{a_{1},\ldots,a_{m}\}$ be the set of atoms. Here an element $a \in M$ is called an {\em atom} if $b^{-1}a \not \in M$ for all $b \in M$. Then
$\Delta$ is a common left and right multiple of $A$. 
\item Let $g \in G$ and $h,h'$ and element of $M \subset G$ which are conjugate to $g$.
 If $\Delta^{-1}g \in M$, then there is a sequence of elements in $M$ 
 \[ h= g_{0} \rightarrow g_{1} \rightarrow \cdots \rightarrow g_{m} = h'\]
 where $\Delta^{-1}g_{i} \in M$ and $g_{i+1}$ is obtained from $g_{i}$ by taking a conjugation by an atom.
\end{enumerate}

The first and second properties are parts of axioms. The last property comes from the solution of the conjugacy problem in Garside groups \cite{bgm}. 
By (2), a conjugation by an atom in a sequence of (3) can be regarded as a cycling operation of words on atoms.
Thus, these properties and Lemma \ref{lem:key} lead to the following results on Hurwitz and HC-equivalences.

\begin{cor}
\label{cor:garside}
Let $(G,M,\Delta)$ be a Garside group and assume that the $M$ is an associated monoid of some finite positive presentation $\mP=\langle \mS \: | \: \mR\rangle $ of $G$ such that 
\begin{enumerate}
\item All relations in $\mR$ are Hurwitz compatible. 
\item The generating set $\mS$ is equal to the set of atoms.
\end{enumerate}

Then for generator $G$-systems $\ba$ and $\ba'$,
\begin{enumerate}
\item $\ba \sim_{H} \ba'$  if and only if $C(\ba) = C(\ba')$.
\item If $\Delta^{-1}C(\ba) \in M_{\mP}^{+}$, then $\ba \sim_{HC} \ba'$ if and only if $C(\ba)$ and $C(\ba')$ are conjugate.
\end{enumerate}
\end{cor}

\begin{exam}
A typical example of a Garside group is an Artin group of finite type $A$, together with the associated monoid $M_{\mP}^{+}$ of standard presentation $\mP= \langle \mS \: | \: \mR \rangle$ and $\Delta= \textrm{The least common multiple of } \mS$. 
For such $(A,M_{\mP}^{+},\Delta)$, the set of atom is identical with the set of standard generators $\mS$, and $(A,M_{\mP}^{+},\Delta)$ satisfies the hypothesis of Corollary \ref{cor:garside}.
Thus, for a generator system $\ba$ of an Artin group of finite type, we can solve not only the Hurwitz equivalence/search problem, but also the harder problems, HC-equivalence/search problems if $\Delta^{-1}C(\ba) \in M_{\mP}^{+}$.
\end{exam}

\section{An algorithm to solve Hurwitz equivalence and Hurwitz search problems}

In this section we present an algorithmic approach to solve the Hurwitz equivalence and Hurwitz search problems.
 
\subsection{Naive algorithm}

First of all, we provide a simple version of  an algorithm to solve Hurwitz equivalence/search problems. This naive version of algorithm still has an advantage compared to the modified algorithm which will be given in Section 4.2. The naive algorithm requires less computations but still works in special cases. More importantly, the naive algorithm stops in finite time.

Let $\mP = \langle\mS \: |\: \mR \rangle$ be a finite positive presentation of a group $G$ such that all relations in $\mR$ are Hurwitz compatible.
We typically consider the finite presentation such that all relations are word-conjugacy relations. 
We further assume that both the word and the conjugacy (search) problems of $G$ are solvable.
Let $\bg = (g_{1},\ldots,g_{m})$ be a generator $G$-system and $\bg' = (g'_{1},\ldots,g'_{m})$ be an arbitrary $G$-system. 

We try to check whether $\bg \sim_{H} \bg'$ or not as follows.

We begin with rather simple tests.
First we compare the Coxeter elements of $\bg$ and $\bg'$. If $C(\bg) \neq C(\bg')$, then $\bg \not \sim_{H}\bg'$. 
Next for each $i$, we check whether there is a permutation $\tau$ of indices such that $g_{i}'$ is conjugate to $g_{\tau(i)}$. If such a permutation does not exist, then again we conclude $\bg \not \sim_{H} \bg'$.

Assume that $\bg$ and $\bg'$ pass these two tests. 
The next step is to construct a new positive presentation $\mP' = \langle \mS' \:  | \: \mR' \rangle $ of $G$ so that both $\bg$ and $\bg'$ are generator $G$-systems with respect to the presentation $\mP'$, and all relations in $\mR'$ are Hurwitz compatible. Such a presentation is constructed as follows.

Let us denote by $g'_{i}= V_{i}^{-1} g_{\tau(i)} V_{i}$ where $V_{i}$ are some fixed words on $\mS \cup \mS^{-1}$ which are computed by solving the conjugacy problem.
Let $L(i) =l(V_{i})$, and write $V_{i}$ as 
\[ V_{i}=a_{n_{1}^{(i)}}^{\varepsilon^{(i)}_{1} }a_{n_{2}^{(i)}}^{\varepsilon^{(i)}_{2} }\cdots a_{n_{L(i)}^{(i)}}^{\varepsilon^{(i)}_{L(i)} }\]
where we put $\mS = \{a_{1},\ldots,a_{M}\}$ and $n_{k}^{(i)} \in \{1,2,\ldots,M\} $, $\varepsilon^{(i)}_{k} \in \{\pm 1\}$. 

We introduce new generators $\{g'_{1},\ldots,g'_{m}\} \cup \{g_{i,j}\}_{i=1,\ldots,m,\; j = 1,\ldots,L(i)-1}$ and new word conjugacy relations $\{\mR_{i,j}\}_{i=1,\ldots,m,\; j = 1,\ldots,L(i)} $ as follows.
For $j=1$, we define the relation $\mR_{i,1}$ as
\[ \mR_{i,1}: \left\{
\begin{array}{l}
 g_{\tau(i)} a_{n_{1}^{(i)}} \equiv a_{n_{1}^{(i)}}g_{i,1}  \;\;\;\;\; (\varepsilon^{(i)}_{1} = +1)\\
   a_{n_{1}^{(i)}} g_{\tau(i)} \equiv g_{i,1}a_{n_{1}^{(i)}} \;\;\;\;\;(\varepsilon^{(i)}_{1} = -1).
\end{array}
\right.
\] 
For $1<j<L(i)$, we define the relation $\mR_{i,j}$ as
\[ \mR_{i,j}:\left\{
\begin{array}{l}
g_{i,j-1}a_{n_{j}^{(i)}} \equiv a_{n_{j}^{(i)}}g_{i,j} \;\;\;\;\; (\varepsilon^{(i)}_{j} = +1)\\
a_{n_{j}^{(i)}}g_{i,j-1} \equiv g_{i,j}a_{n_{j}^{(i)}} \;\;\;\;\; (\varepsilon^{(i)}_{j} = -1).
\end{array}
\right.
\]
Finally, for $j=L(i)$, we define the relation $\mR_{i,L(i)}$ as
  \[ \mR_{i,j}: \left\{
\begin{array}{l} 
g_{i,L(i)-1}a_{n_{L(i)}^{(i)}} \equiv a_{n_{L(i)}^{(i)}}g'_{i} \;\;\;\;\;  (\varepsilon^{(i)}_{L(i)} = +1)\\
a_{n_{L(i)}^{(i)}} g_{i,L(i)-1}\equiv g'_{i}a_{n_{L(i)}^{(i)}}  \;\;\;\;\; (\varepsilon^{(i)}_{L(i)} = -1).
\end{array}
\right.
\]

Let us consider the new positive presentation of $G$
\[ \mP' = \langle \mS \cup \{g_{i,j} \} \: | \: \mR \cup \{\mR_{i,j}\} \rangle. \]
We call this positive presentation $\mP'$ the {\it expanded presentation}.
All of the newly-added relations $\mR_{i,j}$ are word-conjugacy relations, hence Hurwitz-compatible.

Now we reverse the word $W(\bg)^{-1} W(\bg')$ in the presentation $\mP'$.
The reversing procedure stops in finite time because the expanded presentation is finite, homogeneous.
By Theorem \ref{thm:main}, if $W(\bg)^{-1} W(\bg') \curvearrowright \varepsilon$ then we not only conclude $\bg \sim_{H} \bg'$ but also compute a braid $\beta$ such that $\bg \cdot \beta = \bg'$ via Van-Kampen diagrams.

The precise algorithm is given as Algorithm \ref{alg:htestsimple}.  Algorithm \ref{alg:htestsimple} returns {\sf Undecidable} if it fails to determine whether $\bg \sim_{H} \bg'$ or not.

\begin{algorithm}
\caption{:Hurwitz equivalence and search --  Naive algorithm}
\label{alg:htestsimple}
\begin{flushleft}
Input : A finite homogeneous positive presentation $\mP = \langle \mS\: | \:\mR\rangle$ of $G$ such that all relations in $\mR$ are Hurwitz-compatible, a generator $G$-system $\bg= (g_{1},\ldots, g_{m})$, and a $G$-system $\bg' = (g'_{1},\ldots,g'_{m})$.\\
Output: The truth value of $\bg \sim_{H} \bg'$ or {\sf Undecidable}. In case of $\bg \sim_{H} \bg'$, also return a braid $\beta$ such that $\bg \cdot \beta = \bg'$.

\end{flushleft}
\begin{enumerate}
\item If $C(\bg) \neq C(\bg')$, then return \textsf{false}.
\item  Check whether there is a permutation $\tau$ of indices such that $g_{i}'$ is conjugate to $g_{\tau(i)}$. If such a permutation does not exist, then return \textsf{false}.
\item Compute the expanded presentation $\mP'$ of $G$.
\item Check whether $W(\bg)^{-1}W(\bg') \curvearrowright \varepsilon $ or not. If not, then return \textsf{Undecidable}.
\item If $W(\bg)^{-1}W(\bg') \curvearrowright \varepsilon $, then construct a Van-Kampen diagram for $(W(\bg),W(\bg'))$ and compute a braid $\beta$ such that $\bg \cdot \beta = \bg'$ from the Van-Kampen diagram. 
\item Return \textsf{true} and the braid $\beta$.
\end{enumerate}
\end{algorithm}

\subsection{A better Algorithm to solve Hurwitz equivalence and Hurwitz search problems}

In Algorithm \ref{alg:htestsimple}, word reversing of $W(\bg)^{-1}W(\bg') $ is not sufficient to show $\bg \sim_{H} \bg'$, because word reversing might fail to detect the congruence of $W(\bg)$ and $W(\bg')$.
To improve Algorithm \ref{alg:htestsimple} we try to make the expanded presentation complete.
For a complete presentation, the word reversing always detects the congruence so it is more likely to succeed in showing $\bg \sim_{H} \bg'$. 
Moreover, as we will see in Appendix, with additional works one can also try to show the associated monoid embeds into $G$ for left and right complete presentations. Thus, one can also try to obtain the stronger results, the classification of the Hurwitz equivalence classes of generator $G$-systems by using Theorem \ref{thm:main} (3),(4).

The modified algorithm goes as follows. 
The inputs $G = \mP$, $\bg$, $\bg'$ and the first three steps are the same as in Algorithm \ref{alg:htestsimple}.
The next step is the core of the modified algorithm. We try to make the expanded presentation $\mP'$ complete. We slightly modify Algorithm \ref{alg:complete} so that it is more effective for our purposes.
Recall that in the completion procedure, we add a new relation $su \equiv tv$ if $s^{-1}rr^{-1}t \curvearrowright uv^{-1}$ but $(su)^{-1}(tv) \not \curvearrowright_{r} \varepsilon$.

We must check whether the new relation is Hurwitz-equivalent or not because we would like to use Theorem \ref{thm:main}.
Fortunately, adding the relation $su \equiv tv$ does not cause any problem.

\begin{lem}
Assume that $\mP = \langle \mS \: | \: \mR \rangle$ is a positive group presentation such that all relations in $\mR$ are Hurwitz-compatible, and take $s,r,t,u,v$ as above.
Then the relation $su \equiv tv$ is Hurwitz-compatible.
\end{lem}
\begin{proof}
From the reversing sequence $s^{-1}rr^{-1}t \curvearrowright uv^{-1}$, one can construct a diagram which is similar to the Van-Kampen diagram. Indeed, one can find a word $w$ such that this diagram is obtained from two Van-Kampen diagrams of $(su,rw)$ and $(rw,tv)$ by gluing along the path $w$ as shown in Figure \ref{fig:van-kampen2}. Thus, one can find a sequence of words
\[ su = W_{0} \rightarrow W_{1} \rightarrow \cdots \rightarrow W_{i}= rw \rightarrow W_{i+1} \rightarrow W_{k-1} \rightarrow W_{k} = tv\]
where each $W_{j+1}$ is obtained from $W_{j}$ by performing the relation in $\mR$. 
Thus, we can find a braid $\beta$ such that $\bg_{su} \beta = \bg_{tv}$, where $\bg_{su}$, $\bg_{tv}$ are generator $G$-systems whose Coxeter words are $su$, $tv$, so the relation $su \equiv tv$ is Hurwitz-compatible.

\begin{figure}[htbp]
 \begin{center}
\includegraphics[width=40mm]{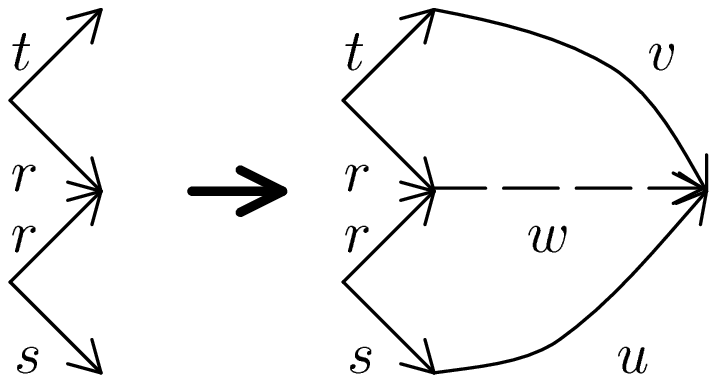}
 \end{center}
 \caption{Van-Kampen-like Diagram from word reversing $s^{-1}rr^{-1}t \curvearrowright uv^{-1}$}
 \label{fig:van-kampen2}
\end{figure}
\end{proof}

Now we consider the case $s = t$, so the relation $su \equiv tv$ is reducible. To detect the Hurwitz equivalences it is better to use finer congruence relations, so it is better to add $u \equiv v$ instead of $su \equiv tv$. Adding the relation $u \equiv v$ also makes the strong cube condition for $s,r,t,u,v$ is satisfied, because $u^{-1}s^{-1}t v \curvearrowright u^{-1}v  \curvearrowright \varepsilon$.
However, one problem occurs. We cannot expect the relation $u \equiv v$ is Hurwitz-compatible. We can add the relation $u \equiv v$ instead of $su \equiv tv$ only if we know the relation $u \equiv v$ is Hurwitz-compatible. In general we cannot know the relation is Hurwitz-compatible, except it is a word-conjugacy relations. Thus we add the relation $u \equiv v$ instead of $su \equiv tv$ if $u\equiv v$ is a word-conjugacy relation.

Summarizing, we modify the completion procedure as follows.
Assume that $s^{-1}rr^{-1}t \curvearrowright uv^{-1}$ but $(su)^{-1}(tv) \not \curvearrowright \varepsilon$. If $s \neq t$, then we add a new relation $su \equiv tv$, which is also Hurwitz-compatible.
If $s = t$, then we need to consider more. If the relation $u \equiv v$ is a word-conjugacy relation, then we add a new relation $u \equiv v$. Otherwise, we add a new relation $su \equiv tv$. The precise description of the modified completion algorithm is given as Algorithm \ref{alg:modifiedcomplete}.

\begin{algorithm}
\caption{:Modified Presentation Completion Algorithm}
\label{alg:modifiedcomplete}
\begin{flushleft}
Input : A finite positive homogeneous presentation $\mP = \langle \mS\: | \:\mR\rangle$ of a group $G$ such that all relations are Hurwitz-compatible.\\
Output: A complete presentation of $G$ such that all relations are Hurwitz-compatible.\\
\end{flushleft}
\begin{enumerate}
\item Compute all pair of words $u,v \in \mS^{*}$ such that $s^{-1}rr^{-1}t \curvearrowright uv^{-1}$ for some $s,r,t \in \mS$.
\item Check $(su)^{-1}(tv) \curvearrowright \varepsilon$ holds for all $u,v$ obtained by Step (1). Assume that $(su)^{-1}(tv) \not\curvearrowright \varepsilon$ for some $u,v$.

\begin{enumerate}
\item If $s \neq t$, 
 then replace the presentation $\mP$ with the new presentation 
\[ 
\langle \mS\: | \:\mR \cup \{ su \equiv tv\} \rangle
\]
and go back to Step (1).
\item If $s=t$, then  replace the presentation $\mP$ with the new presentation 
\[ \left\{\begin{array}{ll}
\langle \mS\: | \:\mR \cup \{ u \equiv v\} \rangle & \textrm{ If } u \equiv v \textrm{ is a word-conjugacy relation} \\
\langle \mS\: | \:\mR \cup \{ su \equiv tv\} \rangle & \textrm{ If } u \equiv v \textrm{ is not a word-conjugacy relation} \\
\end{array}\right.
\]
and go back to Step (1).
\end{enumerate}
\item Stop.
\end{enumerate}
\end{algorithm}

Suppose an Algorithm \ref{alg:modifiedcomplete} terminates and we obtained a complete finite presentation $\overline{\mP'}$. It should be noted that the monoids $M_{\mP'}^{+}$ and $M_{\overline{\mP'}}^{+}$ might be different unlike the usual completion procedure described in Algorithm \ref{alg:complete}.

The rest of steps are the same as the previous algorithm. We reverse the word $W(\bg)^{-1} W(\bg')$ by using the complete presentation $\overline{\mP'}$. 
If $W(\bg)^{-1} W(\bg') \curvearrowright \varepsilon$, then we conclude $\bg \sim_{H} \bg'$ and compute a braid $\beta$ such that $\bg \cdot \beta = \bg'$ via Van-Kampen diagram.

The explicit description of the above algorithm is given as Algorithm \ref{alg:htest}. Algorithm \ref{alg:htest} solves the Hurwitz equivalence problem if possible and returns the value {\sf Undecidable} if it fails to solve. As in Algorithm \ref{alg:htestsimple}, {\sf Undecidable} simply means we can not solve the problem using this algorithm, so it does not imply the problem is undecidable.

\begin{algorithm}
\caption{:Hurwitz equivalence and search -- modified algorithm}
\label{alg:htest}
\begin{flushleft}
Input : A finite positive group presentation $\mP = \langle \mS\: | \:\mR\rangle$ of $G$ such that all relations in $\mR$ are Hurwitz-compatible, a generator $G$-system $\bg= (g_{1},\ldots, g_{m})$, and a $G$-system $\bg' = (g'_{1},\ldots,g'_{m})$.\\
Output: The truth value of $\bg \sim_{H} \bg'$, {\sf Undecidable}. In case of $\bg \sim_{H} \bg'$, then also return a braid $\beta$ such that $\bg \cdot \beta = \bg'$ 

\end{flushleft}
\begin{enumerate}
\item If $C(\bg) \neq C(\bg')$, then return \textsf{false}.
\item Check whether there is a permutation $\tau$ of indices such that $g_{i}'$ is conjugate to $g_{\tau(i)}$. If such a permutation does not exist, then return \textsf{false}.
\item Compute the expanded presentation $\mP'$ of $G$.

\item Make the expanded presentation $\mP'$ complete by using modified completion procedure (Algorithm \ref{alg:modifiedcomplete}).

\item Check whether $W(\bg)^{-1} W(\bg') \curvearrowright \varepsilon$ or not. If $W(\bg)^{-1}W(\bg') \not\curvearrowright \varepsilon$, then return \textsf{undecidable}. 
\item Compute a Van-Kampen diagram for $(W(\bg), W(\bg'))$ using word-reversing.
\item Calculate a braid $\beta$ such that $\bg \cdot \beta = \bg'$ by using the Van-Kampen diagram.
\item Return \textsf{true} and the braid $\beta$.
\end{enumerate}
\end{algorithm}

\begin{rem}
The choice of the expanded presentation $\mP'$ in Algorithm \ref{alg:htestsimple} and \ref{alg:htest} is not unique.
We can use any finite homogeneous positive presentation $\mP''$ of $G$ whose generating set contains $\{g_{1},\ldots, g_{m},g'_{1},\ldots,g'_{m}\}$ and all of whose relations are Hurwitz compatible. Thus, there are many candidates of expanded presentations. Therefore, Algorithm \ref{alg:htestsimple} and Algorithm \ref{alg:htest} have many variations, and some of their variations might be able to solve Hurwitz equivalence problems even if the original Algorithm \ref{alg:htestsimple} and \ref{alg:htest} fail.
\end{rem}

\begin{exam}

Let $G= B_{3}= \langle x,y \: | \: xyx \equiv yxy \rangle$ be the $3$-string braid group with the standard presentation. Let us try to solve Hurwitz equivalence/search problems for two $G$-systems $\bg=(x,x,y,x)$ and $\bg'= (y^{-1}xy,x, y, y^{-1}xy)$ using Algorithm \ref{alg:htest}. 

First we introduce a new generator $z$ and a new word conjugacy relation $yz \equiv xy$, and obtain the expanded presentation  
\[ \mP' = \langle x,y,z \: | \:  xyx \equiv yxy, \: yz \equiv xy \rangle. \] 
Observe that $(xxyx)^{-1}(zxyz) \not \curvearrowright \varepsilon$ in the presentation $\mP'$, because there are no relations of the form $x \cdots \equiv z\cdots$. Hence the naive algorithm, Algorithm \ref{alg:htestsimple} returns {\sf Undecidable}.
Moreover, as we observed in Example \ref{exam:completion}, the usual completion algorithm, Algorithm \ref{alg:complete} does not terminate.

 Now let us apply a modified completion procedure, Algorithm \ref{alg:modifiedcomplete}.
As we have seen in Example \ref{exam:completion}, 
$y^{-1}xx^{-1}y \curvearrowright xyx^{-1}z^{-1}$, but $(yxy)^{-1}yzx \not \curvearrowright \varepsilon$. The relation $yxy \equiv yzx$ is reducible and the reduced relation $xy \equiv zx$ is a word-conjugacy relation. Thus, we add the new relation $xy \equiv zx$ to $\mP'$, and get the presentation
\[ \mP_{1} =  \langle x,y,z \: | \:  xyx \equiv yxy, \: yz \equiv xy \equiv zx \rangle \]

As we have seen in Example \ref{exam:completion}, the presentation $\mP_{1}$ is complete, hence we arrived at the complete presentation $\overline{\mP} = \mP_{1}$. 

Finally, we reverse the word $(xxyx)^{-1}(zxyz)$. As we have seen in Example \ref{exam:van-kampen},  $(xxyx)^{-1}(zxyz) \curvearrowright \varepsilon$ and we get the Van-Kampen diagram of $(xxyx,zxyz)$ (See Figure \ref{fig:van-kampen} again). From the Van-Kampen diagram, we obtain the sequence of words 
\[ xxyx \rightarrow xyxy \rightarrow zxxy \rightarrow zxyz \]
Thus by considering the corresponding braid actions, we conclude that 
\[ \bg \cdot (\sigma_{2}\sigma_{3})(\sigma_{1}^{-1})(\sigma_{3}) = \bg' \]

We also remark that by Example \ref{exam:embeds} given in Appendix, $G_{\overline{\mP}}^{+}$ embeds in $G$. Thus, in this case we actually obtained stronger results, that is, the classification of Hurwitz equivalence classes whose entries are $\{x,y,z\}$.
\end{exam}

\subsection{Remarks on general cases}

We close the paper by giving various remarks to apply our algorithms in general cases.

First of all, results in Section 3 are useful to study Hurwitz equivalence not only for an Artin group, but also for a general group $G$. For an arbitrary $G$-systems $\bg,\bg'$ one can always find a surjective homomorphism from an Artin group $A$ to $G$ and generator $A$-systems $\ba,\ba'$ such that $\ba,\ba'$ are mapped to $\bg,\bg'$.
Thus, if $\ba \sim_{H} \ba'$, which is easily checked by Corollary \ref{cor:artin}, then $\bg \sim_{H} \bg'$. Moreover, if $A$ is a right-angled or finite type Artin group, then sometimes we can show $\bg \sim_{HC} \bg'$.

This kind of a ``lifting" argument is useful to apply our algorithms for general groups.
To try to show $G$-systems $\bg$ and $\bg'$ are Hurwitz equivalent, we consider another group $\widetilde{G}$ with a finite positive homogeneous presentation $\mP = \langle \mS \: | \: \mR \rangle$, a homomorphism $\pi: \widetilde{G} \rightarrow G$, and generator $\widetilde{G}$-systems $\widetilde{\bg}$ and $\widetilde{\bg'}$ such that:
\begin{enumerate}
\item The presentation $\mP = \langle \mP \: | \: \mR \rangle$ satisfies the hypothesis to run Algorithm \ref{alg:htestsimple} or Algorithm \ref{alg:htest}.
\item $\pi$ sends $\widetilde{\bg}, \widetilde{\bg'}$ to $\bg, \bg'$.
\end{enumerate}
Such a group $\widetilde{G}$ can be found, for example, by searching word-conjugacy relations in the entries of $\bg$, $\bg'$. Then we use Algorithm \ref{alg:htestsimple} or Algorithm \ref{alg:htest} to try to show $\widetilde{\bg} \sim_{H} \widetilde{\bg'}$. If Algorithm \ref{alg:htestsimple} or Algorithm \ref{alg:htest} show that  $\widetilde{\bg} \sim_{H}\widetilde{\bg'}$, then we conclude $\bg \sim_{H} \bg'$.

We also remark that the step (2) in Algorithm \ref{alg:modifiedcomplete} can be simplified.
Since the presentation we are considering is homogeneous, so detecting the congruence of Coxeter words $W(\bg)$ and $W(\bg')$ we do not need all relations. It is sufficient to know the relations of length $\leq l$ where $l$ be the length of $\bg$.
Thus, in the step (2) of Algorithm \ref{alg:modifiedcomplete}, if the length of newly-added relations $su \equiv tv$ or $u \equiv v$ become bigger than the length of $\bg$, then we can stop the completion procedure.

Combining these tricks with a ``lifting" argument provides an algorithmic approach to try to show $\bg \sim_{H} \bg'$ for arbitrary $G$-systems $\bg$ and $\bg'$.

\section*{Appendix: Embeddability of the associated monoid}

In Appendix, we give an algorithm to try to show the associated monoid embeds into $G$ which is described in \cite{d2}. 

The word reversing which we used in this paper is actually called a {\it right} word reversing. 
In a similar way, a {\em left} word reversing and the notion of a {\em left} complete presentation are defined. All results described in section 2 holds for left word reversing and left completions as well.
In particular, by modifying Algorithm \ref{alg:complete} or Algorithm \ref{alg:modifiedcomplete} appropriately, one can obtain an algorithm to try to make a finite homogeneous positive presentation left complete, or both right and left complete.

For a finite homogeneous presentation which is both right and left complete, there is a useful criterion for the embeddability of the associated monoid.
\begin{thm}[\cite{d2}, Proposition 7.1]
\label{thm:embed}
Let $\mP=\langle \mS\: | \:\mR\rangle$ be a finite homogeneous presentation which is both left and right complete.
Assume that $\mP$ satisfies the following two conditions $(C)$ and $(E_{r})$.
\begin{itemize}
\item[($C$):] $\mR$ contains no reducible relations.
\item[($E_{r}$):] There exists a set of words $\mS' \subset \mS^{*}$ which contains $\mS$, and for all $u,v \in \mS'$, there exist $u',v' \in \mS'$ such that $(uu')^{-1}(vv')\curvearrowright \varepsilon$.
\end{itemize}
Then the associated monoid $M_{\mP}^{+}$ embeds in $G$.
\end{thm}

The condition ($C$) implies the associated monoid $M_{\mP}^{+}$ is cancellative \cite[Corollary 6.2]{d2}, and the condition ($E_{r}$) implies $M_{\mP}^{+}$ admits a common right multiple \cite[Proposition 6.7]{d2}. 
Using this, one can try to show the embeddability of the associated monoid algorithmically by using Algorithm \ref{alg:embed} below. As in Algorithm \ref{alg:complete}, Algorithm \ref{alg:embed} never stops if $M_{\mP}^{+}$ does not embed in $G$.

\begin{rem}
Algorithm \ref{alg:embed} we present here checks not the condition $(E_{r})$, but the stronger condition $(E'_{r})$ below (See \cite[Remark 6.9]{d2}).
 \begin{itemize}
\item[($E'_{r}$)]: There exists a set of words $\mS' \subset \mS^{*}$ which contains $\mS$, and for all $u,v \in \mS'$, if $u^{-1}v\curvearrowright u'v'^{-1}$, then $u',v' \in \mS'$.
\end{itemize}
($E'_{r}$) implies ($E_{r}$) but in general ($E_{r}$) does not imply ($E'_{r}$).
\end{rem}

\begin{algorithm}
\caption{: Monoid embedding test}
\label{alg:embed}
\begin{flushleft}
Input : A finite homogeneous presentation $\mP = \langle \mS\: | \:\mR\rangle$ of a group $G$ which is both right and left complete and satisfies the condition $(C)$. \\
Output: Return {\sf true} if $M_{\mP}^{+}$ embeds into $G$.\\
\end{flushleft}
\begin{enumerate}
\item Put $\mS'= \mS$.
\item Compute all pairs of words $u',v' \in \mS^{*}$ such that $u^{-1} v\curvearrowright  u' v'^{-1}$ for some $u,v \in \mS'$ and let $\mS'' = \mS' \cup \{u',v'\}$.
\item If $\mS'' = \mS'$, then return {\sf true}. Otherwise, put $\mS'=\mS''$ and go back to step (2).
\end{enumerate}
\end{algorithm}

\begin{exam}
\label{exam:embeds}
Let us consider a presentation $\mP_{1}$ of $B_{3}$ in Example \ref{exam:van-kampen}. 
This presentation is both right and left complete, and contains no reducible relations.
By running Algorithm \ref{alg:embed}, we finally arrive at the set
\[ \mS'' = \{ 1,x,y,z, xy,yz,zx,xyx,yxy, yzx, zxx, yyz \} \]
which is closed under word reversing. That is, for any $u,v \in \mS''$, if $u^{-1}v \curvearrowright u' v'^{-1}$ then $u',v' \in \mS''$. Thus, the associated monoid $M_{\mP_{1}}^{+}$ embeds into $B_{3}$.
\end{exam}

\end{document}